\documentclass[10pt,twoside]{siamltex}
\usepackage{amsfonts,epsfig}
\usepackage{slashbox}
\usepackage{mathdots}
\usepackage{amsmath}
\usepackage{mathrsfs}
\usepackage{algorithm}
\usepackage{algorithmic}
\usepackage{multirow}
\usepackage{color}
\usepackage{booktabs}

\usepackage[font=scriptsize,labelsep=space]{caption}

\newtheorem{example}[theorem]{Example}

\newcommand{\reals}{\mathbb{R}}

\begin{document}



\bibliographystyle{plain}
\title{The eigen-structures of real (skew) circulant matrices with some applications}

\author{
Zhongyun\ Liu\thanks{School of Mathematics and Statistics,
Changsha University of Science and Technology, Changsha 410076, P.
R. China (liuzhongyun@263.net,  mathlife@sina.cn, 649527704@qq.com). } \and Siheng\ Chen\footnotemark[1] \and Weijin\ Xu\footnotemark[1]\and Yulin\ Zhang\thanks{Centro de Matem\'atica,
Universidade do Minho, 4710-057 Braga, Portugal (zhang@math.uminho.pt).}
}

\pagestyle{myheadings} \markboth{Z. Y. \ Liu, et al.}{The real eigen-structure}
\maketitle

\begin{abstract}
The circulant matrices and skew-circulant matrices are two special classes of Toeplitz matrices and play vital roles in the computation of Toeplitz matrices.
In this paper, we focus on real circulant and skew-circulant matrices. We first investigate their real Schur forms, which are closely related to the family of discrete cosine transform (DCT) and discrete sine transform (DST). Using those real Schur forms, we then develop some fast algorithms for computing real circulant, skew-circulant and Toeplitz matrix-real vector multiplications. Also, we develop a DCT-DST version of circulant and skew-circulant splitting (CSCS) iteration for real positive definite Toeplitz systems. Compared with the fast Fourier transform (FFT) version of CSCS iteration, the DCT-DST version is more efficient and saves a half storage. Numerical experiments are presented to illustrate the effectiveness of our method.

\end{abstract}


\vskip 20pt

%
%

\section{\normalsize\bf Introduction}
Recall that a matrix $T= ( t_{jk})_{j, k=0}^{n-1}$ is said to be \emph{Toeplitz} if $ t_{jk} = t_{j-k}$; a matrix $C= ( c_{jk})_{j, k=0}^{n-1}$ is said to be \emph{circulant} if $ c_{jk} = c_{j-k}$ and $c_{-l}= c_{n-l}$ for $1\le l\le n-1$; and a matrix $S= ( s_{jk})_{j, k=0}^{n-1}$ is said to be \emph{skew-circulant} if $ s_{jk} = s_{j-k}$ and $s_{-k}= -s_{n-k}$ for $1\le l\le n-1$.

Toeplitz matrices arise in a variety of applications in mathematics, scientific computing and engineering, for instance, signal processing, algebraic differential equation, time series and control theory, see e.g. \cite{ChanNg96} and a large literature therein. Those applications have motivated both mathematicians and engineers to develop specific algorithms for solving Toeplitz systems for instance \cite{ChanNg96, GolVL96, Daniel2009} and references therein.

The discrete Fourier transform (DFT) matrix $F=(F_{jk})$ is defined by
\begin{equation}\label{Fourier}
F_{jk}=\frac{1}{\sqrt{n}}\omega^{-jk},\ \ j,\ k = 0, 1, \cdots, n-1, \ \mbox{\rm where} \ \omega=\exp(\frac{2\pi}{n}i),\ i=\sqrt{-1}.
\end{equation}

It is known that any circulant matrix $C$ and skew-circulant matrix $S$ possess the following Schur canonical forms \cite{ChanNg96, Ng03, Moody2003}, respectively,
\begin{equation}\label{schur}
   C = F\Lambda F^* ~~~ and ~~~ S = \tilde{F} \tilde{\Lambda} \tilde{F}^*,
\end{equation}
where $\tilde F=DF^*$ is an unitary matrix with $D={\rm diag} \big(1,e^{\frac{\pi}{n}i},\cdots,e^{\frac{(n-1)\pi}{n}i}\big)$, $ \Lambda $ and $ \tilde{\Lambda} $ are diagonal matrices, holding the eigenvalues of $ C $ and $ S $ respectively. Moreover, $ \Lambda $ and $ \tilde{\Lambda}$ can be obtained in $O(n\log n)$ operations by using two FFTs of the first rows of $C$ and $S$, respectively.

Due to the Schur canonical forms (\ref{schur}) of $C$ and $S$, the products $ C\boldsymbol{x}$ or $ S\boldsymbol{x} $ for any vector $\boldsymbol{x}$ can be computed by 3FFTs ( $1$ FFT for computing eigenvalues) in $O(n\log n)$ operations.

Very often, circulant and skew circulant matrices are used to deal with Toeplitz issues. An important property is that a Toeplitz matrix $ T $ can be split  into the following circulant and skew-circulant splitting (CSCS)\cite{Ng03}
\begin{equation}\label{CSsplitting}
 T = C + S
\end{equation}
where $C=(c_{jk})$ is a circulant matrix  and $S=(s_{jk})$ is a skew-circulant matrix, which are defined as follows.
$$
 \begin{array}{lll}
  c_{jk}=
  \left\{
  \begin{array}{cc}
   \frac{1}{2}t_0,                             & {\mbox{\rm if}} \  j=k,   \\[0.2cm]
    \frac{(t_{j-k}+t_{j-k-n})}{2} ,                 & {\mbox{\rm otherwise}},
  \end{array}
  \right. &
\mbox{\rm and} &
  s_{jk}=
  \left\{
  \begin{array}{cc}
   \frac{1}{2}t_0,  & {\mbox{\rm if}} \ j=k,\\[0.2cm]
   \frac{ t_{j-k}-t_{j-k-n}}{2} ,                & {\mbox{\rm otherwise}}.
  \end{array}
  \right.
\end{array}
$$

Actually, due to $T\boldsymbol{x}= C\boldsymbol{x} + S\boldsymbol{x}$, $T\boldsymbol{x}$ can be computed by $6$ FFTs of $n$-vector.
Also, any linear system of equations $C\boldsymbol{x} = \boldsymbol{b}$ ($S\boldsymbol{x} = \boldsymbol{b}$) that contains circulant matrices (skew-circulant matrices) may be quickly solved by using the FFT. However, all operations, due to FFTs, are involved into complex arithmetics, even if $C$ ($S$) and $\boldsymbol{b}$ are real. Now, one may ask when $C$ and $S$ are real, could we find  an analogue of (\ref{schur}) for $C$ and $S$ to avoid complex arithmetics in matrix-vector multiplication? and/or when $C$ ($S$) and $\boldsymbol{b}$ are real, could we develop an algorithm which only involves  real arithmetics for solving $C\boldsymbol{x} = \boldsymbol{b}$ ( $S\boldsymbol{x} = \boldsymbol{b}$)? This is the main motivation of this paper.

The organization of this paper is as follows. In the next section, by exploring the  eigenstructures of $C$ and $S$,  we will give the real Schur forms of the circulant matrix $C$ and the skew-circulant matrix $S$. In Sections 3 and 4, with the real schur forms,  we will develop a real method to fast calculate Toeplitz matrix-vector multiplication and an algorithm based on the CSCS iteration in \cite{Ng03} to solve $T\boldsymbol{x} = \boldsymbol{b}$ by real arithmetics.  Numerical experiments are presented in Section 5 to show the effectiveness of our method.
A brief conclusion and the acknowledgements are finally followed.

%
%
%
\section{\normalsize\bf The Real Schur Forms of Real (Skew) Circulant Matrices}
In the section, making use of the eigen-structures of a real circulant matrix $C$ and a real skew-circulant matrix $S$, we develop their corresponding real Schur forms.

\subsection{\normalsize\bf Preliminaries}

Let's begin with some basic definitions.
For convenience, throughout the paper, we define $J_n$ the permutation matrix of order $n$ with ones on the cross diagonal (bottom left to top right) and zeros elsewhere, and $P_{pq}$ is the $(q-p)\times n$ restriction matrix satisfying $P_{pq}\left (x_{j}\right)_{j=0}^{n-1}=\left(x_{j}\right )_{j=p}^{q-1},\ \ q>p.$
 \begin{definition}{\rm \cite{Heinig1998}}\label{vectorsym}
 A vector $\boldsymbol{x }\in \reals^n $ is said to be symmetric if $J_n \boldsymbol{x} = \boldsymbol{x} $ and skew-symmetric if $J_n \boldsymbol{x} = -\boldsymbol{x} $.
 \end{definition}
\par  Now, let's recall the definitions of DCTs and DSTs. The family of discrete trigonometric transforms consists of $8$ versions of DCTs and corresponding $8$ versions of DSTs \cite{Wang1985, Strang1999, Rao1990}. In this paper, we only need four versions of them which will be used in the sequel.

\begin{definition}\label{DCT}
The{\rm{ DCT-I, DCT-II, DCT-V and DCT-VI}} matrices are defined as follows.
\begin{align*}
\mathscr{C}_{n+1}^{\rm I}&= \sqrt{\frac{n}{2}}\left [\tau_j\tau_k\cos\frac{jk\pi}{n}\right ]_{j, k=0}^{n},\
&\mathscr{C}_n^{\rm V}& = \frac{2}{\sqrt{2n-1}}\left [\tau_j\tau_k\cos\frac{2jk\pi}{2n-1}\right ]_{j, k=0}^{n-1},\\[0.2cm]
\mathscr{C}_n^{\rm II}&= \sqrt{\frac{n}{2}}\left [\tau_j\cos\frac{j(2k+1)\pi}{2n}\right ]_{j, k=0}^{n-1},
&\mathscr{C}_n^{\rm VI}&= \frac{2}{\sqrt{2n-1}}\left [\tau_j\iota_k\cos\frac{j(2k+1)\pi}{2n-1}\right ]_{j, k=0}^{n-1},
\end{align*}
where
$$
 \begin{array}{lll}
  \tau_{l(l=j,\ k)}=
  \left\{
  \begin{array}{lllll}
    1 ,                 & {\mbox{\rm if}}  & l \neq 0  &{\mbox{\rm and}} & l\neq n \\
    \frac{1}{\sqrt{2}}, & {\mbox{\rm if}}  & l = 0     &{\mbox{\rm or}}  & l = n,
  \end{array}
  \right. &
\mbox{\rm and} &
  \iota_k=
  \left\{
  \begin{array}{lll}
    1 ,                & {\mbox{\rm if}} &k \neq n-1  \\
  \frac{1}{\sqrt{2}},  & {\mbox{\rm if}} & k = n-1.
  \end{array}
  \right.
\end{array}
$$
\end{definition}

\begin{definition}
The {\rm{DST-I, DST-II, DST-V and DST-VI}} matrices are defined as follows.
\begin{align*}
\mathscr{S}_{n-1}^{\rm I} &= \sqrt{\frac{2}{n}}\left [\sin\frac{jk\pi}{n}\right ]_{j, k=1}^{n-1},
&\mathscr{S}_{n-1}^{\rm V}&= \frac{2}{\sqrt{2n-1}}\left [\sin\frac{2jk\pi}{2n-1}\right ]_{j, k=1}^{n-1},\\[0.2cm]
\mathscr{S}_n^{\rm II}&= \sqrt{\frac{2}{n}}\left [\tau_j\sin\frac{j(2k-1)\pi}{2n}\right ]_{j, k=1}^{n},
&\mathscr{S}_{n-1}^{\rm VI}&= \frac{2}{\sqrt{2n-1}}\left [\sin\frac{j(2k-1)\pi}{2n-1}\right ]_{j, k=1}^{n-1},
\end{align*}
where $\tau_j$ is defined as in Definition {\rm\ref{DCT}}.
\end{definition}

Note that all those transform matrices are all orthogonal.
\subsection{\normalsize\bf Real Circulant Matrices}\label{real cir}
Let us first to investigate the real eigen-structure of a real circulant matrix $C$.
It is shown in \cite{Moody2003, Karner2003} that the eigenvalues of a real circulant matrix can be arranged in the following order
\begin{enumerate}
\item $\boldsymbol{\lambda} = [\lambda_0, \lambda_1, \cdots, \lambda_{m-1}, \lambda_m, \bar{\lambda}_{m-1}, \cdots, \bar{\lambda}_1]^T$, where $\lambda_0, \lambda_m \in \reals$ and $n=2m$,
    \vskip5pt
 \item $\boldsymbol{\lambda} = [\lambda_0, \lambda_1, \cdots, \lambda_m, \bar{\lambda}_m, \cdots, \bar{\lambda}_1]^T$, where $ \lambda_0 \in \reals $ and $n=2m+1$.
\end{enumerate}

Partitioning $ F^*=[\ \boldsymbol{f}_0, \cdots,\boldsymbol{f}_{n-1}]$, we have $\boldsymbol{f}_{n-k}=\bar{\boldsymbol{f}_k}$. For any eigenvalue $\lambda_k$, $C\boldsymbol{f}_k=\lambda_k \boldsymbol{f}_k$ means
$  C(\boldsymbol{f}_k+\bar{\boldsymbol{f}}_k)=\lambda_k\boldsymbol{f}_k+\bar{\lambda}_k\bar{\boldsymbol{f}}_k\  {\rm and}\
  C(\boldsymbol{f}_k-\bar{\boldsymbol{f}}_k)=\lambda_k\boldsymbol{f}_k-\bar{\lambda}_k\bar{\boldsymbol{f}}_k.$

If we denote $\lambda_k = \alpha_k+i\beta_k$  and $\boldsymbol{f}_k = \hat{\boldsymbol{c}}_k + i\hat{\boldsymbol{s}}_k$,
 then we have

\begin{center} $C[\ \hat{\boldsymbol{c}}_k,\ \hat{\boldsymbol{s}}_k\ ]
=[\ \hat{\boldsymbol{c}}_k,\ \hat{\boldsymbol{s}}_k\ ]
   \left[\begin{array}{cc}
  \alpha_k & \beta_k \\
  -\beta_k & \alpha_k
  \end{array} \right],$\end{center}

 \noindent where $\alpha_k$ and $\beta_k$, $\hat{\boldsymbol{c}}_k$ and $\hat{\boldsymbol{s}}_k$ are the real and pure imaginary parts of $\lambda_k$ and $\boldsymbol{f}_k$, for $k=0, \cdots,n-1$.

Now, we show how to construct an orthogonal matrix $U$ which transforms $C$ into its real Schur form. Notice that $\hat{\boldsymbol{c}}_k=\hat{\boldsymbol{c}}_{n-k}$ and $\hat{\boldsymbol{s}}_k=-\hat{\boldsymbol{s}}_{n-k}$, for $k=0, \cdots,n-1$, so we need only normalize the first half of the vectors $\hat{\boldsymbol{c}}_k$ and $\hat{\boldsymbol{s}}_k$ to get an orthogonal $U$. Namely, $U$ can be chosen as follows,
\begingroup
\renewcommand*{\arraystretch}{1.2}
\begin{equation}\label{thisU}
U= \left\{\begin{array}{ll}
  \left[\ \hat{\boldsymbol{c}}_0, \sqrt{2}\ \hat{\boldsymbol{c}}_1, \cdots, \sqrt{2}\ \hat{\boldsymbol{c}}_{m-1}, \hat{\boldsymbol{c}}_m ,\sqrt{2}\ \hat{\boldsymbol{s}}_{m-1}, \cdots,
\sqrt{2}\ \hat{\boldsymbol{s}}_1  \right],
 & n=2m,\\
  \left[\ \hat{\boldsymbol{c}}_0, \sqrt{2}\ \hat{\boldsymbol{c}}_1, \cdots, \sqrt{2}\ \hat{\boldsymbol{c}}_m, \sqrt{2}\ \hat{\boldsymbol{s}}_m, \cdots,\sqrt{2}\ \hat{\boldsymbol{s}}_1\right],
 & n=2m+1.
 \end{array} \right.
\end{equation}
\endgroup
\indent  A straightforward calculation shows that $U^TCU\equiv\Omega$ is real and has the following structure:
\begingroup
 \begin{equation}\label{cirlschur1}
\Omega_{2m}= \left[
  \begin{array}{c|ccc|c|ccc}
  \alpha_0&        &        &              &        &            &       &         \\ \hline
          &\alpha_1&        &              &        &            &       &\beta_1  \\
          &        & \ddots &              &        &            &\iddots&         \\
          &        &        & \alpha_{m-1} &        &\beta_{m-1} &       &         \\ \hline
          &        &        &              &\alpha_m&            &       &         \\ \hline
          &        &        & -\beta_{m-1} &        &\alpha_{m-1}&       &         \\
          &        & \iddots&              &        &            & \ddots&         \\
          &-\beta_1&        &              &        &            &       &\alpha_1
       \end{array}
   \right]
\end{equation}
\endgroup
or
\begingroup
 \begin{equation}\label{cirlschur2}
  \Omega_{2m+1}= \left[
\begin{array}{c|ccc|ccc}
\alpha_0&         &       &         &         &        &        \\ \hline
        &\alpha_1 &       &         &         &        &\beta_1 \\
        &         &\ddots &         &         &\iddots &        \\
        &         &       &\alpha_m & \beta_m &        &        \\ \hline
        &         &       &-\beta_m & \alpha_m&        &        \\
        &         &\iddots&         &         & \ddots &        \\
        &-\beta_1 &       &         &         &        &\alpha_1
         \end{array}
  \right],
 \end{equation}
\endgroup
which can be transformed into the {\it real Schur canonical form} by a permutation. Therefore we also refer to (\ref{cirlschur1}) or (\ref{cirlschur2}) as the {\it real Schur form} of $C$.
This leads to the following theorem.
\begin{theorem}[\bf{Real Schur form of real circulant matrices}]\label{realform1}
Let $U$ be defined as in {\rm{(\ref{thisU})}}. If the circulant matrix $C$ is real, then $ U^T CU = \Omega $ is the real Schur form of $C$.
\end{theorem}
\begin{proof}
The proof can be  directly given by the above analysis and thus omitted.
\end{proof}
\subsection{\normalsize\bf Real Skew-Circulant Matrices}
Similarly, the eigenvalues of a real skew-circulant matrix $S$ can be arranged in the following order
\begin{enumerate}
 \item $\tilde{\boldsymbol{\lambda}} = [\tilde{\lambda}_0, \cdots, \tilde{\lambda}_{m-1}, \bar{\tilde{\lambda}}_{m-1}, \cdots, \bar{\tilde{\lambda}}_0]^T$, \ where\  $n=2m $,
 \vskip5pt
 \item $\tilde{\boldsymbol{\lambda}} = [\tilde{\lambda}_0, \cdots, \tilde{\lambda}_{m-1},\tilde{\lambda}_m, \bar{\tilde{\lambda}}_{m-1}, \cdots, \bar{\tilde{\lambda}}_0]^T$, \ where\ $ \tilde{\lambda}_m \in \reals $ \ and \ $n=2m+1$.
 \end{enumerate}

The next procedure is very much like section \ref{real cir}.
Let $\tilde{F}^*=[\ \tilde{\boldsymbol{f}}_0, \cdots, \tilde{\boldsymbol{f}}_{n-1}]$. Analogously, denoting $\tilde{\lambda}_k=\tilde{\alpha}_k+i\tilde{\beta}_k$ and $\tilde{\boldsymbol{f}}_k=\tilde{\boldsymbol{c}}_k+i\tilde{\boldsymbol{s}}_k$, then we have
$  S[\ \tilde{\boldsymbol{c}}_k,\ \tilde{\boldsymbol{s}}_k\ ]
  =[\ \tilde{\boldsymbol{c}}_k,\ \tilde{\boldsymbol{s}}_k\ ]
  \left[\begin{array}{cc}
  \tilde{\alpha}_k & \tilde{\beta}_k \\
  -\tilde{\beta}_k & \tilde{\alpha}_k
  \end{array} \right].$

Due to $\tilde{\boldsymbol{c}}_k=\tilde{\boldsymbol{c}}_{n-k-1} $ and $\tilde{\boldsymbol{s}}_k=-\tilde{\boldsymbol{s}}_{n-k-1}$, for $k = 0,  \cdots, n-1$,
we can choose an orthogonal matrix $\tilde{U}$ as follows,
\begingroup
\renewcommand*{\arraystretch}{1.2}
\begin{equation}\label{thisV}
\tilde{U}=
   \left\{  \begin{array}{ll}
   \left[\ \sqrt{2}\ \tilde{\boldsymbol{c}}_0, \cdots,\sqrt{2}\ \tilde{\boldsymbol{c}}_{m-1},
  \sqrt{2}\ \tilde{\boldsymbol{s}}_{m-1}, \cdots,\sqrt{2}\ \tilde{\boldsymbol{s}}_0  \right ],
 & {\mbox{\rm if}} \ \ n=2m,\\
  \left[\ \sqrt{2}\ \tilde{\boldsymbol{c}}_0,  \cdots,\sqrt{2}\ \tilde{\boldsymbol{c}}_{m-1}, \tilde{\boldsymbol{c}}_m,
 \sqrt{2}\ \tilde{\boldsymbol{s}}_{m-1},  \cdots,\sqrt{2}\ \tilde{\boldsymbol{s}}_0  \right ],
 & {\mbox{\rm if}} \ \  n=2m+1,
 \end{array}
 \right.
\end{equation}
\endgroup
which makes $\tilde{U}^TS\tilde{U}\equiv\Sigma $ being of the following structure.
 \begingroup
 \begin{equation}\label{skewschur1}
 \Sigma_{2m}=
 \left[
     \begin{array}{ccc|ccc}
 \tilde{\alpha}_0&       &                    &                   &        &\tilde{\beta}_0   \\
                 &\ddots &                    &                   &\iddots &       \\
                 &       &\tilde{\alpha}_{m-1}&\tilde{\beta}_{m-1}&        &      \\ \hline
                 &       &-\tilde{\beta}_{m-1}&\tilde{\alpha}_{m-1}&       &       \\
                 &\iddots&                    &                   & \ddots &       \\
 -\tilde{\beta}_0&       &                    &                   &        & \tilde{\alpha}_0
      \end{array}
      \right]
  \end{equation}
  \endgroup
 or
 \begingroup
 \begin{equation}\label{skewschur2}
      \Sigma_{2m+1}=
      \left[
      \begin{array}{ccc|c|ccc}
 \tilde{\alpha}_0&&&&&&\tilde{\beta}_0\\
&\ddots&&&&\iddots&\\
&&\tilde{\alpha}_{m-1}&&\tilde{\beta}_{m-1}&&\\\hline
&&&\tilde{\alpha}_m&&&\\\hline
&&-\tilde{\beta}_{m-1}&&\tilde{\alpha}_{m-1}&&\\
&\iddots&&&&\ddots&\\
-\tilde{\beta}_0&&&&&&\tilde{\alpha}_0
         \end{array}
         \right].
 \end{equation}
\endgroup
Similarly, we refer to (\ref{skewschur1}) or (\ref{skewschur2}) as the {\it real Schur form} of $S$.
Based on the above analysis, we conclude the following theorem.

\begin{theorem}[\textbf{Real Schur form of real skew-circulant matrices}]\label{realform2}
Let $\tilde{U}$ be defined as in {\rm{(\ref{thisV})}}. If the skew-circulant matrix $S$ is real, then $\tilde{U}^T S\tilde{U} = \Sigma$ is the real Schur form of $S$.
\end{theorem}

We remark here that different from $C = F \Lambda F^*$ (i.e., $C$ is factorized into the product of $3$ complex matrices), Theorems \ref{realform1} - \ref{realform2} tell us that both $C$ and $S$ can be factorized into the products of  $3$ real matrices, respectively. This fact allows us to fast calculate matrix-vector multiplication and solve $C\boldsymbol{x} = \boldsymbol{b}$ and $S\boldsymbol{x} = \boldsymbol{b}$ by only real operations. In the next two sections, we will derive this strategy.


%
%
%
\section{\normalsize\bf  Fast Matrix-vector Multiplication}
In this section, we first reduce the matrices $U$ and $\tilde{U}$ into simpler forms by exploiting the structures of $U$ and $\tilde{U}$, then show how to fast compute $\Omega$ in (\ref{cirlschur1}) or (\ref{cirlschur2})  and $\Sigma$ in (\ref{skewschur1}) or (\ref{skewschur2}), and finally develop fast algorithms for computing $C\boldsymbol{x}$, $S\boldsymbol{x}$ and $T\boldsymbol{x}$.

Recall Definition \ref{vectorsym}, $P_{1n}\hat{\boldsymbol{c}}_k$ (remove the first entry of $\hat{\boldsymbol{c}}_k$) is a symmetric vector and $P_{1n}\hat{\boldsymbol{s}}_k$ (remove the first entry of $\hat{\boldsymbol{s}}_k$) is a skew-symmetric vector. Notice that if we delete the first row of $U$, then the first $m+1$ columns of the submatrix are all  symmetric, and the last columns are all skew-symmetric. Therefore, the $U$ of (\ref{thisU}) can be partitioned into the following form

 \begingroup
\renewcommand*{\arraystretch}{1.2}
\setlength{\arraycolsep}{0.9pt}
$
U_{2m}=
  \left[
  \begin{array}{ccc}
  \sigma_1 \boldsymbol{q}_{m+1}^T  & &\textbf{0}     \\[0.1cm]
  \hat{\mathscr{C}}                & &-\mathscr{S}_{m-1}^{\rm I}  J_{m-1}      \\[0.1cm]
  \sigma_1 \boldsymbol{v}_{m+1}^T  & &\textbf{0}  \\[0.1cm]
  J_{m-1}\hat{\mathscr{C}}         & &J_{m-1}\mathscr{S}_{m-1}^{\rm I}  J_{m-1}
  \end{array}
  \right]
 $ and
   \begingroup
\renewcommand*{\arraystretch}{0.9}
\setlength{\arraycolsep}{0.9pt}
$
U_{2m+1}=
  \left[
   \begin{array}{cc}
    \sigma_1 \boldsymbol{p}_{m+1}^T    & \textbf{0} \\[0.1cm]
    \hat{\mathscr{C}}                  & -\mathscr{S}_{m}^{\rm V}  J_m \\[0.1cm]
    J_m\hat{\mathscr{C}}               & J_m\mathscr{S}_{m}^{\rm V}  J_m
  \end{array}
  \right]$\\
where
$\sigma_1 = \sqrt{\frac{2}{n}}$, \
$\sigma_2 =\frac{1}{\sqrt{2}}$,\
$\boldsymbol{p}_{m+1}=(\frac{1}{\sqrt{2}}, 1, \cdots, 1)^T$,\
$\boldsymbol{q}_{m+1}=(\frac{1}{\sqrt{2}}, 1, \cdots, 1, \frac{1}{\sqrt{2}})^T$,\
$\boldsymbol{v}_{m+1} =(\frac{1}{\sqrt{2}}, -1, \cdots, (-1)^{m-1}, \frac{(-1)^{m}}{\sqrt{2}})^T$,
and
$$ \hat{\mathscr{C}}=\left\{
 \begin{array}{ll}
  \sigma_2  P_{1,m} \mathscr{C}_{m+1}^{\rm I}\in \reals^{(m-1) \times (m+1)},
 &\quad {\mbox{\rm if}} \ \ n=2m,\\[0.1cm]
  \sigma_2 P_{1,m+1} \mathscr{C}_{m+1}^{\rm V} \ \in \reals^{m \times (m+1)},
 & \quad {\mbox{\rm if}} \ \  n=2m+1.
 \end{array}
\right.$$
Similarly, the $\tilde{U}$ of (\ref{thisV}) can be partitioned into the following form
   \begingroup
\renewcommand*{\arraystretch}{1.1}
$$
\begin{array}{ccc}
\tilde{U}_{2m}=
   \left[
   \begin{array}{cc}
    \sigma_1\boldsymbol{e}^T   & \textbf{0} \\[0.1cm]
    \tilde{\mathscr{C}}        & \mathscr{S}_{m}^{\rm II} J_m \\[0.1cm]
   \textbf{0}                  & \sigma_1 \boldsymbol{u}_{m}^T J_m\\[0.1cm]
    -J_{m-1}\tilde{\mathscr{C}}&  J_{m-1}\mathscr{S}_{m}^{\rm II} J_m
  \end{array}
  \right]
  & {\mbox{\rm and }} &
\tilde{U}_{2m+1} =
   \left[
   \begin{array}{cc}
    \sigma_1\boldsymbol{p}_{m+1}^T & \textbf{0} \\[0.1cm]
    \tilde{\mathscr{C}}            & \mathscr{S}_{m}^{\rm VI} J_m \\[0.1cm]
    -J_m\tilde{\mathscr{C}}        & J_m\mathscr{S}_{m}^{\rm VI} J_m
  \end{array}
  \right ]
\end{array}
$$
\endgroup
where $\boldsymbol{u}_{m}=(1, -1, \cdots,(-1)^{m-1})^T$ and
$$
\tilde{\mathscr{C}}=
\left\{
 \begin{array}{ll}
\sigma_2P_{1m}\mathscr{C}_{m}^{\rm II}\ \in \reals^{(m-1)\times m}, & \quad {\mbox{\rm if}}\ \ n=2m,\\[0.1cm]
\sigma_2P_{1m}\mathscr{C}_{m+1}^{\rm VI}\ \in \reals^{m\times (m+1)}, & \quad {\mbox{\rm if}}\ \ n=2m+1.\\
 \end{array}
\right.
$$
Now we construct an orthogonal matrix $Q$ of the form
\begin{equation}\label{Q}
  Q = \left\{
  \begin{array}{ll}
       \frac{1}{\sqrt{2}}
       \left[
       \begin{array}{cccc}
        \sqrt{2} &         &          &        \\
                 & I_{m-1} &          & J_{m-1}  \\
                 &         & \sqrt{2} &         \\
                 & -J_{m-1}&          & I_{m-1}
        \end{array}
        \right],
& {\mbox{\rm if }}\   n=2m,\\ \\
     \frac{1}{\sqrt{2}}
     \left[
     \begin{array}{ccc}
        \sqrt{2} &      &      \\
                 & I_m  & J_m  \\
                 & -J_m & I_m
        \end{array}
        \right],
&  {\mbox{\rm if }}\  n=2m+1.
  \end{array}  \right.
\end{equation}
Then we can get the following simple formulas.
\begin{theorem}\label{QU}
Let $U$, $\tilde{U}$ and $Q$ be defined as in {\rm(\ref{thisU})}, {\rm(\ref{thisV})} and {\rm(\ref{Q})}, respectively. Then we have
\begin{equation}\label{U}
QU = \left\{
\begin{array}{ll}
\left[
     \begin{array}{cc}
         \mathscr{C}_{m+1}^{\rm I} &  \\
                     & J_m \mathscr{S}_{m-1}^{\rm I} J_m \\
    \end{array}
  \right],
& {\mbox{\rm if}} \ \  n=2m,\\ \\
\left[
       \begin{array}{cc}
         \mathscr{C}_{m+1}^{\rm V} &  \\
                        & J_m \mathscr{S}_m^{\rm V} J_m \\
       \end{array}
\right],
& {\mbox{\rm if}} \ \ n=2m+1.
\end{array}
\right.
\end{equation}
and
\begin{equation}\label{V}
Q^T\tilde{U} = \left\{
\begin{array}{ll}
 \left[
  \begin{array}{cc}
         \mathscr{C}_{m}^{\rm II} &  \\
                        & J_m\mathscr{S}_{m}^{\rm II} J_m \\
       \end{array}
       \right],
& {\mbox{\rm if}} \ \  n=2m.\\ \\
      \left[
      \begin{array}{ll}
         \mathscr{C}_{m+1}^{\rm VI} &  \\
                          & J_m\mathscr{S}_{m}^{\rm VI}  J_m \\
       \end{array}
       \right],
&  {\mbox{\rm if}} \ \ n=2m+1.
\end{array}
 \right.
\end{equation}
\end{theorem}
\begin{proof}
The proof of this theorem can be completed by a tedious straightforward calculation and thus omitted.
\end{proof}

The Theorem \ref{QU} provides us a fast way to the calculation of $U\boldsymbol{x}$, it can be obtained by $1$ DCT-I of $(m+1)$-vector and $1$ DST-I of $(m-1)$-vector if $n=2m$,  and $1$ DCT-V of $(m+1)$-vector and $1$ DST-V of $m$-vector if $n=2m+1$. The product $\tilde{U}\boldsymbol{x}$ can be obtained by a similar mode by employing the second and sixth versions of DCT and DST.

The entries of $\Omega$ and $\Sigma$ can be computed by
\begin{equation}\label{omega-sigama}
  \Omega U^T \boldsymbol{e}_1=(QU)^TQC\boldsymbol{e}_1 \ \ {\rm{and}}\ \
  \Sigma \tilde{U}^T \boldsymbol{e}_1=(Q^T\tilde{U})^TQ^TS\boldsymbol{e}_1,
\end{equation}
where $\boldsymbol{e}_1=(1, 0, \cdots,0)^T$.

The left-hand sides of (\ref{omega-sigama}) are as follows, respectively,
   \begingroup
\renewcommand*{\arraystretch}{1.1}
\begin{equation}\label{left-omega}
\sqrt{n}\Omega U^T \boldsymbol{e}_1=\left\{
\begin{array}{ll}
 (\frac{\alpha_0}{\sqrt{2}}, \alpha_1, \cdots,  \alpha_{m-1}, \frac{\alpha_m}{\sqrt{2}}, -\beta_{m-1}, \cdots, -\beta_1 )^T, & n=2m,\\[0.15cm]
  (\frac{\alpha_0}{\sqrt{2}}, \alpha_1, \cdots,  \alpha_m, -\beta_m, \cdots, -\beta_1)^T, &n=2m+1 ,
\end{array}
\right.
\end{equation}
\endgroup
and
   \begingroup
\renewcommand*{\arraystretch}{1.1}
\begin{equation}\label{left-sigama}
\sqrt{n}\Sigma \tilde{U}^T\boldsymbol{e}_1
=\left\{
\begin{array}{ll}
  (\tilde{\alpha}_0, \cdots,  \tilde{\alpha}_{m-1}, -\tilde{\beta}_{m-1}, \cdots, -\tilde{\beta}_0 )^T, & n=2m,\\[0.15 cm]
  (\tilde{\alpha}_0, \cdots, \tilde{\alpha}_{m-1},  \frac{\tilde{\alpha}_m}{\sqrt{2}}, -\tilde{\beta}_{m-1}, \cdots, -\tilde{\beta}_0)^T, &n=2m+1.
\end{array}
\right.
\end{equation}
\endgroup
This means we only need $1$ DCT and $1$ DST of about $\frac{n}{2}$-vector to get $\Omega$ or $\Sigma$.

Now, we show the calculations of $C\boldsymbol{x}$ and $S\boldsymbol{x}$ for  any real $n$-vector $\boldsymbol{x}$ using DCT and DST.

 According to Theorem \ref{realform1} and Theorem \ref{QU}, $C\boldsymbol{x}$ can be easily obtained by three DSTs and three DCTs (version I or V) of about $\frac{n}{2}$-vector. As for the storage required, we need one temporary $n$-vector and an extra $n$-vector for storing $\Omega$. In fact, we don't need to compute and store $U$ and $\Omega$ explicitly.
It can be written as the following Algorithm \ref{alg1}.
\begin{algorithm}[!htbp]
\caption{ To calculate $C\boldsymbol{x}$ }\label{alg1}
\begin{algorithmic}[1] 
\STATE Compute $\boldsymbol{v}=Q\boldsymbol{c}_1$ directly.
\STATE Compute $\hat{\boldsymbol{v}}=(QU)^T\boldsymbol{v}$ by DCT and DST.
\STATE Form $\Omega$.
\STATE Compute $\boldsymbol{y}_1=Q\boldsymbol{x}$ directly.
\STATE Compute $\boldsymbol{y}_2=(QU)^T\boldsymbol{y}_1$ by DCT and DST.
\STATE Compute $\boldsymbol{y}_3=\Omega\boldsymbol{y}_2$ directly.
\STATE Compute $\boldsymbol{y}_4=(QU)\boldsymbol{y}_3$ by DCT and DST.
\STATE Compute $Q^T\boldsymbol{y}_4$, i.e., $C\boldsymbol{x}$.
\end{algorithmic}
\end{algorithm}

If we compute $C\boldsymbol{x}$ by (\ref{schur}), it requires three FFTs of $n$-vector, and one temporary complex $n$-vector and an extra complex $n$-vector for storing $\Lambda$, equivalently, two temporary real $n$-vectors and two extra real $n$-vectors for storing $\Lambda$.

Similarly, according to Theorem \ref{realform2} and Theorem \ref{QU}, we develop the following Algorithm \ref{alg2} for computing the product $S\boldsymbol{x}$,  which can be obtained by three DSTs and three DCTs (version II or VI) of about $\frac{n}{2}$-vector.
\begin{algorithm}[!htbp]
\caption{ To calculate $S\boldsymbol{x}$}\label{alg2}
\begin{algorithmic}[1] 
\STATE Compute $\boldsymbol{u}=Q^T\boldsymbol{s}_1$ directly.
\STATE Compute $\hat{\boldsymbol{u}}=(Q^T\tilde{U})^T\boldsymbol{u}$ by DCT and DST.
\STATE Form $\Sigma$.
\STATE Compute $\boldsymbol{z}_1=Q^T\boldsymbol{x}$ directly.
\STATE Compute $\boldsymbol{z}_2=(Q^T\tilde{U})^T\boldsymbol{z}_1$ by DCT and DST.
\STATE Compute $\boldsymbol{z}_3=\Sigma\boldsymbol{z}_2$ directly.
\STATE Compute $\boldsymbol{z}_4=(Q^T\tilde{U})\boldsymbol{y}_3$ by DCT and DST.
\STATE Compute $Q\boldsymbol{z}_4$, i.e., $S\boldsymbol{x}$.
\end{algorithmic}
\end{algorithm}
\par From (\ref{CSsplitting}), we have that a Toeplitz matrix-vector multiplication $T\boldsymbol{x} = C\boldsymbol{x} + S\boldsymbol{x}$ can be fast calculated by employing the Algorithms \ref{alg1} - \ref{alg2}.

%
%
%
%
%
\section{\normalsize\bf Solving $T\boldsymbol{x} = \boldsymbol{b}$ by the CSCS iteration}
Consider the iterative solution to a large scale system of linear equations
\begin{equation}\label{Equation}
  T\boldsymbol{x}=\boldsymbol{b},
\end{equation}
where $T\in \reals^{n \times n}$ is a Toeplitz matrix and $ \boldsymbol{b} \in \reals^{n}$.

Based on the splitting (\ref{CSsplitting}), Ng proposed in \cite{Ng03} the following CSCS iteration for solving (\ref{Equation}).

 {\bf The CSCS iteration: }{\it Given an initial guess $\boldsymbol{x}^{(0)}$, for $k=0, 1,  \cdots, $ until
$\{\boldsymbol{x}^{(k)}\}$ converges, compute
\begin{equation}\label{CSequation}
\left\{
\begin{aligned}
(\theta I+C)\boldsymbol{x}^{(k+\frac{1}{2})}&=(\theta I-S)\boldsymbol{x}^{(k)}+\boldsymbol{b},\\[0.2cm]
(\theta I+S)\boldsymbol{x}^{(k+1)}&=(\theta I-C)\boldsymbol{x}^{(k+\frac{1}{2})}+\boldsymbol{b},
\end{aligned}
\right.
\end{equation}
where $\theta$ is a given positive constant.}

It is shown in \cite{Ng03} that the CSCS iteration converges unconditionally, if both $C$ and $S$ are positive definite.

Then applying (\ref{schur}) to (\ref{CSequation}), we get

{\bf The FFT version of CSCS iteration: }{\it Given an initial guess $\boldsymbol{x}^{(0)}$, for $k=0, 1, \cdots $, until $\{\boldsymbol{x}^{(k)}\}$ converges, compute
\begin{equation}\label{CSfourier}
\left\{\begin{aligned}
F(\theta I+\Lambda)F^*\boldsymbol{x}^{(k+\frac{1}{2})}
&=\tilde{F}(\theta I-\tilde{\Lambda})\tilde{F}^*\boldsymbol{x}^{(k)}+\boldsymbol{b},\\[0.25cm]
\tilde{F}(\theta I+\tilde{\Lambda})\tilde{F}^*\boldsymbol{x}^{(k+1)}
&=F(\theta I-\Lambda)F^*\boldsymbol{x}^{(k+\frac{1}{2})}+\boldsymbol{b}.
\end{aligned}\right.
\end{equation}
where $\theta$ is a given positive constant.}

In the preparatory stage, two FFTs of $n$-vector for computing $\Lambda$ and $\tilde{\Lambda}$ are required. In the iterative stage, six FFTs of $n$-vector for solving (\ref{CSfourier}) are needed. Therefore the computational complexity is $O(n\log n)$ complex flops at each iteration. However, all operations, due to FFTs, are involved into complex arithmetics, even if $T$ and $\boldsymbol{b}$ are real.

In this section, we develop the  DCT-DST version of (\ref{CSequation}) based on the DCT and DST. Also, we compare the computational cost of our version with the FFT version of the CSCS iteration (\ref{CSfourier}).

Note by Theorem \ref{realform1} and Theorem \ref{realform2} that the CSCS iteration (\ref{CSequation}) can be reformulated as the following form.
\begin{equation}\label{CSschur}
   \left\{
   \begin{aligned}
   U(\theta I+\Omega)U^T\boldsymbol{ x}^{(k+\frac{1}{2})}=&   \tilde{U}(\theta I-\Sigma)\tilde{U}^T\boldsymbol{x}^{(k)}+\boldsymbol{b},\\[0.25cm]
 \tilde{U}(\theta I+\Sigma)\tilde{U}^T \boldsymbol{x}^{(k+1)}=& U(\theta I-\Omega)U^T\boldsymbol{x}^{(k+\frac{1}{2})}+\boldsymbol{b}.
   \end{aligned}
   \right.
\end{equation}
The equation (\ref{CSschur}) can be further reduced into a simpler form due to Theorem \ref{QU}. For example, consider the case $n=2m$ (The odd case is similar to the even case), we have the following version,

 {\bf  The DCT-DST version of CSCS iteration: }{\it Given an initial guess $\boldsymbol{x}^{(0)} \in \reals^n $,  compute $\boldsymbol{x}^{(k)}$,  for $k=0, 1, \cdots $, until $\{\boldsymbol{x}^{(k)}\}$ converges:
 \begingroup
\renewcommand*{\arraystretch}{0.9}
\setlength{\arraycolsep}{0.9pt}
\begin{equation}\label{CSschur_q}
   \left\{
   \begin{aligned}
   &Q^T\left[
     \begin{array}{cc}
       \mathscr{C}_{m+1}^{\rm I} &  \\
        & \mathscr{S}_{m-1}^{\rm I} \\
     \end{array}
   \right]
   (\theta I+\Omega)
   \left[
     \begin{array}{cc}
       \mathscr{C}_{m+1}^{\rm I}& \\
       &\mathscr{S}_{m-1}^{\rm I} \\
     \end{array}\right]
     (Q\boldsymbol{x}^{(k+\frac{1}{2})})\\
   \ =\ &Q\left[
     \begin{array}{cc}
       \mathscr{C}_{m}^{\rm II} &  \\
        & \mathscr{S}_{m}^{\rm II} \\
     \end{array}
   \right]
   (\theta I-\Sigma)
   \left[
     \begin{array}{cc}
       \mathscr{C}_{m}^{\rm II} &  \\
        & \mathscr{S}_{m}^{\rm II} \\
     \end{array}\right]^T(Q^T\boldsymbol{x}^{(k)})+\boldsymbol{b},
\\[0.25cm]
   &Q\left[
     \begin{array}{cc}
       \mathscr{C}_{m}^{\rm II} &  \\
        & \mathscr{S}_{m}^{\rm II} \\
     \end{array}
   \right]
   (\theta I+\Sigma)
   \left[
     \begin{array}{cc}
       \mathscr{C}_{m}^{\rm II}&\\
       &\mathscr{S}_{m}^{\rm II} \\
     \end{array}
   \right]^T
   (Q^T\boldsymbol{x}^{(k+1)})\\
   \ =\ &Q^T\left[
     \begin{array}{cc}
       \mathscr{C}_{m+1}^{\rm I} &  \\
        & \mathscr{S}_{m-1}^{\rm I} \\
     \end{array}
   \right]
   (\theta I-\Omega)
   \left[
     \begin{array}{cc}
       \mathscr{C}_{m+1}^{\rm I}&\\
      & \mathscr{S}_{m-1}^{\rm I} \\
     \end{array}
     \right](Q\boldsymbol{x}^{(k+\frac{1}{2})})+\boldsymbol{b},
   \end{aligned}
   \right.
\end{equation}
\endgroup
where $\theta$ is a given constant.}

We emphasize here that our version has the same convergence rate and optimal parameter as the CSCS iteration does.

{\bf The computational complexity.}
The iteration (\ref{CSschur_q}) consists of the preparatory stage and computational stage. In preparatory stage, we only need to calculate the matrices $\Omega$ and $\Sigma$ which can be obtained by two DCTs and DSTs of about $n/2$-vector, see (\ref{omega-sigama}), (\ref{left-omega}) and (\ref{left-sigama}). In the computational stage, we need to compute the inverses of $\theta I + \Omega$ and $\theta I + \Sigma$, respectively.
Because the matrices $\theta I + \Omega$ and $\theta I + \Sigma$ are of special structure, their inverses also keep the same structure as the original matrices and can be easily obtained which only cost $O(n)$ flops.
The remaining operations are all matrix-vector multiplications. It takes $O(n)$ flops to calculate the products $Q\boldsymbol{v}$, $Q^T\boldsymbol{v}$, $(\theta I + \Omega)^{-1}\boldsymbol{v}$, $(\theta I - \Omega)\boldsymbol{v}$, $(\theta I + \Sigma)^{-1}\boldsymbol{v}$ and $(\theta I - \Sigma)\boldsymbol{v}$.
Therefore, the main cost of each iteration is six DCTs and six DSTs of about $n/2$-vectors.
As is well known, the complexity of DCT-I, DCT-II and DST-II of an $n$-vector is $\frac{1}{2}n\log n$ multiplications and $\frac{3}{2}n\log n$ additions, while DST-I of an $n$-vector requires $\frac{1}{2}n\log n$ multiplications and $2n\log n$ additions, see \cite{Britanak2007, Yip1988,  Martin1984}.
Thus the {\it computational complexity} of one iteration for solving (\ref{CSschur_q}) is $\frac{25}{2}n\log n$.
Note that to perform a FFT of an $n$-vector requires  $5n\log n$ flops, see, for example, \cite{GolVL96}. Therefore, the {\it computational complexity} of one iteration for solving (\ref{CSfourier}) is $30n\log n$. This means our method can save about half  operations as compared with the FFT version of the CSCS iteration.

%
%

\section{\normalsize\bf Numerical Examples}

All the numerical tests are done on a Founder desktop PC with quad-core Intel(R)  Core(TM) i7-4790 CPU 3.60 GHz with MATLAB 7.11.0(R2010b).

In all tests, we take the right-hand side $\boldsymbol{b}$ of (\ref{Equation}) to be $(1,  \cdots, 1)^T$ and the initial guess $\boldsymbol{x}^{(0)}$ to be the zero vector.
All tests are performed with double precision, and terminated when the current iterate satisfies
$\frac{\parallel \boldsymbol{r}^{(k)}\parallel_2}{\parallel \boldsymbol{r}^{(0)}\parallel_2}\leq 10^{-7}$,
or when the number of iterations is over 500,
where $\boldsymbol{r}^{(k)}$ is the residual vector of the system (\ref{Equation}) at the current iterate $\boldsymbol{x}^{(k)}$, and $\boldsymbol{r}^{(0)}$ is the initial one.

To show the effectiveness of our version (\ref{CSschur_q}), we give some comparisons of the elapsed CPU time among the iteration (\ref{CSschur_q}),  the FFT version of CSCS iteration (\ref{CSfourier}), and the AHSS iteration \cite{Fang2010} for solving real positive definite Toeplitz systems, whose generating functions are listed as follows.
\begin{example}\label{ex2}{\rm\cite{Ng03}}
$t_k=(1+\mid k\mid )^{-p}, k=0, \ \ \pm 1, \cdots\pm (n-1).$

\end{example}

\begin{example}\label{ex7}{\rm\cite{Fang2010}}
 $f(x)=5+x^2+2cos(3x)+i(x+sinx),\ x\in [-\pi, \pi]$.
\end{example}

\begin{example}\label{ex5}{\rm\cite{Fang2010}}
 $f(x)=10+8cosx+i2sin(5x),\ x\in [-\pi, \pi]$.
\end{example}

The generated Toeplitz matrix $T_n$ is symmetric positive definite in Example \ref{ex2}, and nonsymmetric positive definite in Examples \ref{ex7}-\ref{ex5}.
Therefore, all versions of the CSCS iteration are convergent unconditionally.

In all tables, we denote the order of the matrix
$T$, the parameter of the iteration, spectral radius of the corresponding iteration matrix, the number of iterations by $n$, $\theta$\
 \footnote{We emphasize here that the true  optimal parameter $\theta$ in the CSCS iteration is difficult to get, the parameters $\theta$ in all tables is experimentally approximately optimal.
 The parameters $\theta_1$  and $\theta_2$ in Table \ref{Table3} and \ref{Table4} are obtained from Corollary 4.1 in \cite{Fang2010}. },
$\rho$, $N$, respectively.
Also, we denote the elapsed CPU times ($10^{-2}$ seconds) of the FFT version of  CSCS iteration,  the DCT-DST version of CSCS iteration, the AHSS iteration based on FFT
by $t_s$, $t_r$ and $t_a$, respectively.

\begin{table}[htb]
\scriptsize
\caption{Comparison between (\ref{CSfourier}) and (\ref{CSschur_q}) for Example \ref{ex2} }\label{Table1}
\centering
\begin{tabular}{c cccc cccc}
\toprule
\multirow{2}{*}{\centering $n$} &
\multicolumn{4}{c}{ p=0.9}&
\multicolumn{4}{c}{ p=1.1}\\
\cmidrule(r){2-5}
\cmidrule(l){6-9}
 &$\theta$ &$N$ & $t_s$ &$t_r$ &
 $\theta$ &$N$ &  $t_s$ & $t_r$\\
\midrule
4000&1.985&21&58.576&34.222&1.465&14&42.548&24.320\\
6000&2.095&22&140.69&71.187&1.555&14&86.545&46.634\\
8000&2.175&22&227.67&124.07&1.545&14&152.99&83.585\\
\bottomrule
\end{tabular}
\end{table}
\begin{table}[htb]
\scriptsize
\caption{Comparison between (\ref{CSfourier}) and (\ref{CSschur_q})  for Examples \ref{ex7}-\ref{ex5}}\label{Table2}
\centering
\begin{tabular}{c  cccc cccc}
\toprule
 \multirow{2}{*}{\centering $n$}&
 \multicolumn{4}{c}{ Example \ref{ex7}}&
 \multicolumn{4}{c}{ Example \ref{ex5}}\\
\cmidrule(r){2-5}
\cmidrule(l){6-9}
&$\theta$ &$N$ &  $t_s$&$t_r$ &
$\theta$ &$N$ &  $t_s$ & $t_r$\\
\midrule
4000&3.680&5&17.133&10.108&3.890&9&29.907&18.628\\
6000&3.720&5&38.635&19.397&3.940&9&59.054&31.359\\
8000&3.705&5&66.787&32.603&3.925&8&95.722&49.814\\
\bottomrule
\end{tabular}
\end{table}
\begin{table}[htb]
\scriptsize
\caption{Comparison between (\ref{CSschur_q}) and AHSS iteration for Example \ref{ex7} }\label{Table3}
\centering
\begin{tabular}{c cccc cccc}
\toprule
 \multirow{2}{*}{\centering $n$}
 &\multicolumn{3}{c}{DCT-DST version}
 &\multicolumn{3}{c}{ AHSS iteration}\\
\cmidrule(r){2-4}
\cmidrule(l){5-7}
 &$(\theta, \rho)$ & $N$ & $t_r$&
$ (\theta_1, \theta_2, \rho)$  &$N$ &$t_a$ \\
\midrule
256   &(3.595, 0.1554)&6&1.7019&(7.1280, 7.1484, 0.2782)&9&2.0794 \\
512   &(3.765, 0.1656)&6&1.8638&(7.1444, 7.1550, 0.2813)&9&6.7190\\
1024 &(3.865, 0.1718)&6&2.4665&(7.1532, 7.1586, 0.2830)&9&32.539\\
\bottomrule
\end{tabular}
\end{table}
\begin{table}[htb]
\scriptsize
\caption{Comparison between (\ref{CSschur_q}) and AHSS iteration for Example \ref{ex5} }\label{Table4}
\centering
\begin{tabular}{c cccc cccc}
\toprule
 \multirow{2}{*}{\centering $n$}
 &\multicolumn{3}{c}{ DCT-DST version}
 &\multicolumn{3}{c}{ AHSS iteration}\\
\cmidrule(r){2-4}
\cmidrule(l){5-7}
 &$(\theta ,\rho)$ & $N$ & $t_r$ &
 $(\theta_1,\theta_2,\rho)$  &$N$ &$t_a$ \\
\midrule
256   &(3.585, 0.2806)&9&1.8882&(6.0035, 6.0005, 0.4916)&23&4.0687 \\
512   &(3.665, 0.2878)&9&2.0877&(6.0009, 6.0001, 0.4917)&23&11.418\\
1024 &(3.735, 0.2971)&9&2.8626&(6.0002, 6.0000, 0.4917)&23&70.218\\
\bottomrule
\end{tabular}
\end{table}

The  computational efficiency  of two versions of CSCS iteration for solving Examples \ref{ex2}- \ref{ex5}  is shown in the first two tables.
As mentioned in the previous section, the spectral radius  of iterative matrices of the two versions are the same, hence the numbers of iteration of them also remain the same.
From Tables \ref{Table1}-\ref{Table2}, we can see that when the order $n$ of the matrix $T$ becomes much larger, the iteration (\ref{CSschur_q}) works nearly twice as fast as the FFT version (\ref{CSfourier}).

Also, we give a comparison of  the computational efficiency  between the iteration (\ref{CSschur_q}) and the AHSS iteration\footnote{The AHSS iteration in \cite{Fang2010} involves two parameters and is faster than the HSS iteration proposed by Bai, et al. in \cite{Golub2003}. }
developed by Gu \cite{Gu2009} and Chen \cite{Fang2010} who have tailored the HSS iteration proposed in \cite{Golub2003} for real positive definite Toeplitz systems. Recall that any Toeplitz matrix $T$ admits a Hermitian and skew-Hermitian splitting \cite{Gu2009, Fang2010} $T=H +\tilde{S}$, where $H=\frac{1}{2}(T+T^*)$ and $\tilde{S}=\frac{1}{2}(T-T^*)$. Then the resulting HSS iteration for solving real positive definite Toeplitz system is

{\bf The HSS iteration: }{\it Given an initial guess $\boldsymbol{x}^{(0)}$, for $k=0, 1, \cdots $ until
$\{\boldsymbol{x}^{(k)}\}$ converges, compute
\begin{equation}\label{HSeqution}
\left\{
\begin{aligned}
(\theta I+H)\boldsymbol{x}^{(k+\frac{1}{2})}&=(\theta I-\tilde{S})\boldsymbol{x}^{(k)}+\boldsymbol{b},\\
(\theta I+\tilde{S})\boldsymbol{x}^{(k+1)}&=(\theta I-H)\boldsymbol{x}^{(k+\frac{1}{2})}+\boldsymbol{b},
\end{aligned}
\right.
\end{equation}
where $\theta$ is a given positive constant.}

Because $H$ is a symmetric Toeplitz matrix and $\tilde{S}$ is a skew-symmetric Toeplitz matrix, by using a unitary similarity transformation, each system of (\ref{HSeqution}) can be reduced into two subsystems with about half sizes. Thus, one need to solve three Toeplitz-plus-Hankel subsystems. The complexity of each iteration  is at least $O(n\log^2 n)$ if the superfast direct method is employed, and may be reduced to $O(\frac{n}{2}\log \frac{n}{2})$ if a preconditioned conjugate gradient method \cite{Fang2010} is used. However, a good preconditioner is not easy to get. Moreover, the employ of FFTs makes the complex operations be involved for real system (\ref{Equation}).

We remark here that the AHSS iteration does not work for symmetric case, Example \ref{ex7} and \ref{ex5} are used only. From the Tables \ref{Table3}-\ref{Table4}, we can see that the spectral radius of the corresponding iteration matrix of (\ref{CSschur_q}) is much smaller than that of the AHSS iteration.
Therefore, the number of iterations required of (\ref{CSschur_q})  is also less than that of the AHSS  iteration.
In particular, our iterative methods perform more efficiently for large $n$.

We must point out that in our tests, the functions \ $dct(\boldsymbol{x})$ and \ $dst(\boldsymbol{x})$, i.e.,  DCT-II and DST-I, are employed. In fact, in MATLAB, the DCT-II and DST-I are FFT-based algorithms for speedy computation. So the computational efficiency showed in practice is far lower than one in theoretical analysis, see the analysis of computational complexity in preceding sections.
In case the real \ $dct(\boldsymbol{x})$ and \ $dst(\boldsymbol{x})$  are implemented in MATLAB in the future, it can be expected that our method will behave more efficiently.

\section{\normalsize\bf Conclusion}
In this paper, by exploiting the special eigen-structure  of a real circulant matrix $C$ and a real skew-circulant matrix $S$,  we  get their real Schur forms. Then, by means of DCT-DST, we further reduce the orthogonal matrices $U$ and $\tilde{U}$ into simpler forms. As their applications, we first develop two  new fast algorithms for computing matrix-vector multiplications $C\boldsymbol{x}$ and $S\boldsymbol{x}$, where only real arithmetics are involved. Then we reformulate the CSCS iteration for $T\boldsymbol{x} = \boldsymbol{b}$, and get
the DCT-DST version of CSCS iteration (\ref{CSschur_q}), which only involves real arithmetics, is highly parallelizable and can be implemented on multiprocessors efficiently. Finally, some numerical examples are presented to show the reduction on the elapsed CPU times, compared with the iteration (\ref{CSfourier}) and the AHSS  iteration.
In fact, our method can save a half storage  and about half operations compared to the FFT version (\ref{CSfourier}).
Our proposed method shows obvious  advantage especially in high order cases.

\section{\normalsize\bf Acknowledgement}
The authors would like to thank the supports of the National Natural Science Foundation of China under Grant No. 11371075, the Hunan Key Laboratory of mathematical modeling and analysis in engineering, and the Portuguese Funds through FCT-Fundac$\tilde{a}$o para a Ci$\hat{e}$ncia, within the Project UID/ MAT/ 00013/2013.


\end{document}